\documentclass[leqno]{siamltex}
\usepackage{amsmath,amssymb}
\usepackage{graphicx}
\usepackage{enumerate}           
\usepackage{url}
\usepackage{color}
\usepackage{tikz}
\usepackage{subfig}
\usetikzlibrary{arrows,intersections}
\usepackage{pgfplots}
\pgfplotsset{compat=1.12}
\usepackage{caption}

\def \dis{\displaystyle}

\def \R{\mathbb{R}} 

\def \N{\mathbb{N}}

\def \N0{\mathbb{N}_0}

\def \d{\delta}
\def \e{\varepsilon}

\def \W{\Omega}
\def \phi{\varphi}

\def \12{\dis\frac{1}{2}}

\def \1{\mathbbm{1}}

\def \<{\left<}
\def \>{\right>}

\def \mA{\mathcal{A}}

\def \mAze{\mathcal{A}_0^\e}
\def \Th{\mathcal{T}_h}

\def \EI{\mathcal{E}_h^I}

\def \Eh{{\mathcal E}_h}

\def \dv{\cdot\nu_{e}}
\def \mL{\mathcal{L}}

\def \mLh{\mathcal{L}_{h}}
\def \mLh{\mathcal{L}_h}

\def \mLzh{\mathcal{L}_{0,h}}
\def \grad{\nabla}
\def \lss{\lesssim}

\def \wto{\rightharpoonup}

\let\div\undefined
\DeclareMathOperator*{\div}{div}

\def \dx[#1]{\ensuremath{\operatorname{d}\!{#1}}}
\newtheorem{defn}{Definition}
\numberwithin{defn}{section}

\numberwithin{remark}{section}

\begin{document}
	
\title{$H^1$-norm error estimate for a nonstandard finite element approximation of second-order linear elliptic
	 PDEs in non-divergence form.}
 \markboth{X. FENG and S. SCHNAKE}{FINITE ELEMENT METHOD FOR NON-DIVERGENCE FORM ELLIPTIC PDEs}

\author{Xiaobing Feng\thanks{Department of Mathematics, The University of Tennessee, Knoxville, TN 37996, U.S.A. (xfeng@math.utk.edu).}
\and
	Stefan Schnake\thanks{Department of Mathematics, The University of Oklahoma, Norman, OK 73019, U.S.A. (sschnake@ou.edu).}
}

\maketitle
\date{\today}

\begin{abstract}
	This paper establishes the optimal $H^1$-norm error estimate for a nonstandard finite element 
	method for approximating $H^2$ strong solutions of second order linear elliptic PDEs in 
	non-divergence 	form with continuous coefficients. To circumvent the difficulty of lacking 
	an effective duality argument for this 	class of PDEs, a new analysis technique is introduced; 
	the crux of it is to establish an $H^1$-norm stability estimate for the finite element approximation   
	operator which mimics a similar estimate for the underlying PDE operator recently established by the 
	authors and its proof is based on a freezing coefficient technique and a topological argument. 
	Moreover, both the $H^1$-norm stability and  error estimate also hold for the linear finite element method. 
\end{abstract}

\begin{AMS}
	65N30, 65N12, 35J25
\end{AMS}
  
\section{Introduction}\label{introduction}  

This paper is concerned with  finite element approximations of the following linear elliptic PDE in non-divergence form:
\begin{align} \label{eqn1.1}
\begin{split} 
\mL u:= -A:D^2u &= f \qquad \text{ in } \W, \\ 
u &= 0 \qquad \text{ on } \partial\W, 
\end{split}
\end{align} 
where $\W\subset\R^n (n=2,3)$ is an open bounded domain, $f\in L^2(\W)$, and $A\in[C^0(\overline{\W})]^{n\times n}$ 
is uniformly positive definite.  
The above non-divergence form PDEs can be seen in several applications -- most notably from stochastic optimal control, game theory, and mathematical finance \cite{SP:Flem}.  Moreover, non-divergence PDEs relate to several second-order fully nonlinear 
PDEs such as the Hamilton-Jacobi-Bellman equation, Issac's equations, and the Monge-Amp\`ere equation \cite{AM:CG,Crandall_Ishii_Lions}.

Because of the non-divergence structure, it is not easy to develop convergent Galerkin-type methods for problem \eqref{eqn1.1}.
As expected, the inherent difficultly is that we cannot perform integration by parts on the non-divergence term $A:D^2u$.  
This issue could be avoided if $A$ is sufficiently smooth as then we could rewrite $A:D^2u$ as the sum of a divergence 
form diffusion term and a first-order advection term, namely, $-A:D^2u=-\div(A\nabla u) + \div(A) \cdot\nabla u$. However, 
when $A$ is only continuous, we cannot perform this rewriting. Due to these challenges, only a few convergent numerical methods 
have been developed so far for problem \eqref{eqn1.1} in the literature -- see  \cite{Feng2017SS,Feng2018SS,AX:WW,AX:FN,SmearsSuli,NochettoZhang16}. Many of these works aim at approximating 
the strong solution $u\in H^2(\W)\cap H_0^1(\W)$ that satisfies \eqref{eqn1.1} almost everywhere in $\overline{\W}$.  

In this paper we further study the $C^0$ finite element method proposed by Feng, Neilan, and Hennings \cite{AX:FN} 
which is defined by seeking $u_h$ in a finite element space $V_h$ on a triangular mesh $\Th$ with interior skeleton $\EI$, 
such that
\begin{align} \label{eqn1.1.1}
 \sum_{T\in\Th} \int_T -A:D^2 u_h v_h\dx[x] + \sum_{e\in\EI} \int_e [A\grad u_h\dv] v_h \dx[S] = \int_\W fv_h \dx[x]
\end{align}
for any $v_h\in V_h$.   The authors proved the well-posedness of \eqref{eqn1.1.1} in addition to stability estimate 
\begin{align} \label{eqn1.1.2}
 \| u_h \|_{H^2(\W)} \leq C\|f\|_{L^2(\W)},
\end{align}
and the optimal $H^2$-norm error estimate
\begin{align} \label{eqn1.1.3}
 \| u - u_h\|_{H^2(\W)} \leq Ch^{\min\{r+1,s\}-1}\|u\|_{H^s(\W)},
\end{align}
with $r\geq 2$ being the polynomial degree of finite element functions. Many of the other Galerkin-type methods share 
estimates similar to \eqref{eqn1.1.2} and \eqref{eqn1.1.3} \cite{Feng2017SS,AX:WW,AX:FN,SmearsSuli,Feng2018SS}. These 
estimates arise from the framework provided by the operator $\mL:H^2(\W)\to L^2(\W)$ and we note that the energy
space for the strong solution of problem \eqref{eqn1.1} is $H^2(\W)$ (or $W^{2,p}(\W)$ in general).  However, it is natural 
to ask whether optimal order error estimates can be obtained for $u-u_h$ in lower order norms such as the $H^1$ and $L^2$-norm.  Numerical tests in \cite{AX:FN} indicate that $u_h$ obtained by \eqref{eqn1.1.2} yields optimal error estimates 
in $H^1$ and $L^2$-nrom.  However, none of the existing works (\cite{Feng2017SS,AX:WW,AX:FN,SmearsSuli,Feng2018SS})
prove an optimal $H^1$-norm error estimate without assuming additional regularity to the coefficient matrix $A$.

A standard technique to obtain error estimates in lower order norms  is the well-known Aubin-Nitsche duality argument, 
which is in fact the only general tool for such a job.  
Below we motivate that the using this technique will most likely fail for \eqref{eqn1.1.1}.  Our motivation stems from \cite{Feng2017SS} where 
an IP-DG counterpart of \eqref{eqn1.1.1} was developed using the interior-penalty discontinuous Galkerin (IP-DG) framework;
namely, given an interior skeleton $\EI$ and a full skeleton $\Eh$ find $u^h$ in a DG finite element space $V^h$ of $\Th$ 
such that
\begin{align}\label{eqndg}
\begin{split}
&-\sum_{T\in\Th}\int_T (A:D^2u^h)v^h \dx[x] + \sum_{e\in\EI}\int_e[A\grad u^h\dv]\{v^h\}\dx[S] \\
&\qquad-\e \sum_{e\in\Eh}\int_e \{A\grad v^h\dv\}[u^h]\dx[S] + \sum_{e\in\Eh}\int_e \frac{\gamma_e}{h_e}[u^h][v^h]\dx[S] = \int_\W fv^h \dx[x] 
\end{split}
\end{align}
for any $v^h\in V^h$.  We refer the reader to \cite{Feng2017SS} for the detailed derivation and analysis of \eqref{eqndg}; 
here we only write down the formulation to show that \eqref{eqndg} gives three methods dependent on the choice of $\e$.  
Recall that $\e=1,0,-1$ yield the symmetrically induced (SIP-DG), incompletely induced (IIP-DG), and non-symmetrically 
induced (NIP-DG) methods respectively. The numerical tests of \cite{Feng2017SS} show that $L^2$-norm error estimates 
are not always optimal.  The sub-optimality in $L^2$ should be expected as incomplete and non-symmetric methods do not 
yield optimal $L^2$-norm error estimates even for divergence form PDEs such as $-\Delta u  = f$.  More specifically, 
a duality argument will fail because the bilinear form given by the left-hand side of \eqref{eqndg} is not symmetric 
for $\e=0$ and $\e=-1$. If a hypothetical duality technique were to be used to show optimal $H^1$-norm error estimates, 
we should expect the same sub-optimality as in the $L^2$ case for the IIP-DG and NIP-DG methods. On the other hand, the 
numerical tests show the $H^1$-norm error estimates are always optimal for any $\e=1,0,-1$.  Thus these tests suggest 
that the duality augment will probably not yield optimal $H^1$ estimates.  

To circumvent the difficulty of lacking an effective duality argument, we take a different route by showing that 
the solution of the finite element method \eqref{eqn1.1.1} satisfies the following $H^1$ stability estimate:
\begin{align} \label{eqn1.1.4}
\|u_h\|_{H^1(\W)} \leq C\|f\|_{H^{-1}(\W)}.  
\end{align}
With \eqref{eqn1.1.4} in hand, optimal order error estimates in $H^1$-norm immediately follow.  
The motivation for \eqref{eqn1.1.4} arises from \cite{Feng2018SS} where the authors have recently 
shown that $\mL:H^1(\W)\to H^{-1}(\W)$ with stability estimate
\begin{align} \label{eqn1.1.5}
\|u\|_{H^1(\W)} \leq C\|f\|_{H^{-1}(\W)}.  
\end{align}
We see that \eqref{eqn1.1.4} is the discrete analogue to \eqref{eqn1.1.5}, and it will be proved by adapting 
the freezing coefficient technique of \cite{Feng2018SS} at the discrete level. 

Moreover, the numerical experiments in \cite{AX:FN,Feng2017SS} suggest that the $C^0$ finite element method defined by  \eqref{eqn1.1.1} and IP-DG method 
defined by \eqref{eqndg} are well-posed and converge for the linear finite element.  Such a result cannot be shown using the $H^2$-norm stability estimate \eqref{eqn1.1.2} 
because  linear finite element functions cannot accurately approximate $H^1$ functions in a 
discrete $H^2$-norm.  
As a consequence, the authors in \cite{AX:FN,Feng2017SS} restricted their analysis to quadratic 
elements and greater.  
However, since \eqref{eqn1.1.4} based in an $H^1$-norm, we additionally show that \eqref{eqn1.1.1} is well-posed and 
converges optimally in the $H^1$-norm for the linear finite element. 

The rest of the paper is organized as follows.  In Section \ref{prelim} we defines the PDE problem and notation 
as well as introduce an auxiliary lemma.  In Section \ref{sectionstab}, we prove \eqref{eqn1.1.4} by first 
considering the case $A:=A_0$ (constant coefficient) and then extending the result to the case of continuous coefficient $A$.  For the case of linear elements, we must additionally use a nonstandard duality argument.
We then derive the desired optimal order error estimate for the  finite element method in the $H^1$-norm.  

\section{Preliminaries}\label{prelim}  

\subsection{Notation}\label{notation} Let $\W$ be an open bounded polygonal domain in $\R^d$.  We use the notation $L^p(\W)$ and $H^k(\W):=W^{k,2}(\W)$ be the standard Lebesgue and Sobolev spaces with appropriate norms, and let $H^{-1}(\W)$ be the dual space of $H_0^1(\W)$.  Let $(\cdot,\cdot)_D$ be the standard $L^2$ inner-product on $D$ with $(\cdot,\cdot):=(\cdot,\cdot)_\W$.

Given $h>0$, we say $a\lss b$ if there is a constant independent of $h$ such that $a\leq Cb$.  Let $\Th$ be a quasi-uniform and shape regular triangulation of $\W$ with interior skeleton $\EI$.  Given $e\in\EI$, let $\nu_e$ be the (well-defined) unit edge normal vector such that $e=\partial T^+\cap\partial T^-$ with $\nu_e=\nu_{T^+}=-\nu_{T^-}$ where $\nu_{T^{\pm}}$ are the unit normal vectors for $\partial T^\pm$ respectively.  We then define the jump of a function $u$ on an edge $e\in\EI$ by
\begin{align*}
[u] = u\big|_{T^+} - u\big|_{T^-}
\end{align*}
where $T^{\pm}$ is well-defined through $\nu_e$.  Lastly, define $\left<\cdot,\cdot\right>=\left<\cdot,\cdot\right>_e$ to be the $L^2$ inner product on $e$ for any $e\in\EI$.

We now define the specific function spaces used in this paper.   Set $V_h := V_h^{r}$ be the $C^0$ Lagrange finite element space of polynomial degree $r$.  In addition, let $H^2(\Th)$ be the broken $H^2$ Sobolev space defined by
\begin{align*}
H^2(\Th) = \{v\in L^2(\W) : v\in H^2(T) \ \forall T\in\Th\},
\end{align*}
and let $H_h^2(\W) := H^2(\Th)\cap H_0^1(\W)$.
Given a subdomain $D\subseteq \W$, define the space 
\begin{align*}
H_h^2(D):=\{v\in H_h^2(\W) : v\big|_{\W\setminus D} \equiv 0\}
\end{align*}
with norm
\begin{align*}
\|w\|_{H_h^2(D)} = \sum_{T\in\Th}\|D^2w\|_{L^2(T\cap D)} + \bigg(\sum_{e\in\EI}h_e^{-1}\|[\grad w]\|_{L^2(e\cap\overline{D})}\bigg)^{\frac12}.
\end{align*}
Clearly we have $V_h\subset H_h^2(\W)$.  
Additionally define $V_h(D)\subset V_h$ by
\begin{align}
V_h(D) = \{ v_h\in V_h : v_h\big|_{\W\setminus D}\equiv 0\}.
\end{align}
We note from \cite{AX:FN} that the space $V_h(D)$ is non-trivial provided $\mbox{diam}(D)>2h$.  
Also, define the discrete dual-norms $\|\cdot\|_{L_h^2(D)}$ and $\|\cdot\|_{H_h^{-1}(D)}$ by
\begin{align*}
\|w\|_{L_h^2(D)} &= \sup_{v_h\in V_h(D)\setminus\{0\}}\frac{(w,v_h)_D}{\|v_h\|_{L^2(D)}}, \text{ and}  \\
\|w\|_{H_h^{-1}(D)} &= \sup_{v_h\in V_h(D)\setminus\{0\}}\frac{(w,v_h)_D}{\|\grad v_h\|_{L^2(D)}}.
\end{align*}
Let $\mathcal{P}_h:L^2(\W)\to V_h(D)$ be the $L^2$ projection onto $V_h(D)$ given by 
\begin{align}\label{l2proj}
(\mathcal{P}_h w,v_h)_D = (w,v_h)_D \ \forall v_h\in V_h(D).  
\end{align}
From \cite{crouzeix1987stability}, we have 
\begin{align}\label{l2est}
\|\mathcal{P}_hw\|_{H^1(D)} \lss \|\grad w\|_{L^2(D)}
\end{align}
for any $w\in H_0^1(D)$.  Thus, from \eqref{l2proj} and \eqref{l2est} we obtain
\begin{align}\label{normequiv}
\|w_h\|_{H^{-1}(D)} &= \sup_{v\in H_0^1(D)\setminus\{0\}}\frac{(w_h,v)_D}{\|\grad v\|_{L^2(D)}} = \sup_{v\in H_0^1(D)\setminus\{0\}}\frac{(w_h,\mathcal{P}_h v)_D}{\|\grad v\|_{L^2(D)}} \\
&\lss \sup_{v\in H_0^1(D)\setminus\{0\}}\frac{(w_h,\mathcal{P}_h v)_D}{\|\grad \mathcal{P}_hv\|_{L^2(D)}} \leq \|w_h\|_{H_h^{-1}(D)} \nonumber
\end{align}
for any $w_h\in V_h(D)$.  

\subsection{PDE problem and finite element method}\label{pdeprob}

We first introduce the problem. Let $f\in L^2(\W)$.  In addition, we assume $A\in [C^0(\overline\W)]^{d\times d}$ 
is symmetric and uniformly positive definite, that is, there is $0<\lambda\leq\Lambda$ such that
\begin{align}\label{eigenvalue}
\lambda |\xi|^2 \leq \xi^TA(x)\xi \leq \Lambda |\xi|^2 \quad\forall x\in\overline\W,\xi\in\R^d.
\end{align}
We seek to approximate the unique strong solution $u\in H^2(\W)\cap H_0^1(\W)$ to
\begin{align*} \label{prob}\tag{$P$}
\begin{split} 
\mL u:= -A:D^2u &= f \qquad \text{ in } \W, \\ 
u &= 0 \qquad \text{ on } \partial\W.  
\end{split}
\end{align*} \addtocounter{equation}{1}
that satisfies \eqref{prob} a.e.\ in $\W$.  Here $A:D^2u$ is the matrix inner product give by
\begin{align*}
A:D^2u = \sum_{i,j=1}^{d} a_{ij}u_{x_ix_j}.
\end{align*}
In addition to the invertibility of $\mL$, we also have the stability estimates
\begin{align}
\|u\|_{H^2(\W)} &\lss \|\mL u\|_{L^2(\W)} \label{eqn1.2},\\
\|u\|_{H^1(\W)} &\lss \|\mL u\|_{H^{-1}(\W)} \label{eqn1.3}.
\end{align}
The well-posedness of \eqref{eqn1.1} and the estimates (\ref{eqn1.2}-\ref{eqn1.3}) are guaranteed provided $\partial\W\in C^{1,1}$. \cite[Chapter 9]{Sp:DT} and \cite{Feng2018SS}.  

As per \cite{AX:FN}, define the discrete linear operator $\mLh:V_h\to V_h$ by 
\begin{align} \label{eqnlh}
(\mLh w_h,v_h) = (-A:D_h^2w_h,v_h) + \sum_{e\in\EI}\left<[A\grad w_h\dv],v_h \right>_e,
\end{align}
where $D_h^2$ is the piecewise defined Hessian on every $T\in\Th$.

We can now define the nonstandard finite element method for problem \eqref{eqn1.1}.

\begin{defn} \label{defnFEmethod}
We define the $C^0$ finite element method for \eqref{eqn1.1} as seeking $u_h\in V_h$ such that
\begin{align} \label{eqnc1fe}
(\mLh u_h,v_h) = (f,v_h)  \qquad\forall v_h\in V_h.
\end{align}
\end{defn} 
From \cite{AX:FN}, there is a unique solution $u_h\in V_h$ to \eqref{eqnc1fe} for $r\geq 2$ with stability estimate
\[
 \|w_h\|_{H_h^2(\W)} \lss \|\mLh w_h\|_{L_h^2(\W)}
\]
for all $w_h\in V_h$.  
Note that we can extend $\mLh:H^2(\Th)\cap H_0^1(\W)\to H^{-1}(\W)$ by
\begin{align} \label{eqnlhdual}
(\mLh w,v) = (-A:D_h^2w,v) + \sum_{e\in\EI}\left<[A\grad w\dv],v \right>_e \quad\forall v\in H_0^1(\W).
\end{align}

Lastly, we quote a super-approximation result from \cite{AX:FN}.

\begin{lemma}\label{lemsuperapprox}
Let $I_h$ be the $C^0$ nodal finite element interpolant onto $V_h$ and $\eta\in C^{\infty}(\W)$ with $\|\eta\|_{W^{j,\infty}}=\mathcal{O}(d^{-j})$ for some $h\leq d<1$.  Then for any subdomain $D\subseteq\W$ with inscribed radius larger larger than $3h$ we have the following:
\begin{align*}
\|\grad(\eta v_h-I_h(\eta v_h))\|_{L^2(D)} &\lss \frac{1}{d}\|v_h\|_{L^2(D)}.
\end{align*}
\end{lemma}

\section{Discrete $H^1$-norm stability estimate}  \label{sectionstab}

Our goal in this section is to prove a similar analogue of \eqref{eqn1.3} for our discrete operator $\mLh$, that is
\begin{align}\label{eqnh1stabtemp}
	\|\grad w_h\|_{L^2(\W)} \lss \|\mLh w_h\|_{H_h^{-1}(\W)}
\end{align}
for any $w_h\in V_h$.  To achieve this, we follow the freezing coefficient technique on the discrete level as seen in \cite{Feng2018SS,Feng2017SS,AX:FN}; however, because we already have the existence and uniqueness of $u_h$ to \eqref{eqnc1fe}, 
we can bypass the rather lengthy and technical nonstandard duality argument given in those works.  The freezing coefficient
technique exploits the fact that since $A$ is continuous, it is essentially a constant locally. For $A_0$ constant, 
we may represent the non-divergence operator $A_0:D^2 u$ as a divergence operator $\div(A_0\grad u)$.  Thus,  using 
standard elliptic theory we arrive at
at \eqref{eqnh1stabtemp} for $A_0$.  Using the continuity of $A$, we may achieve a local version of \eqref{eqnh1stabtemp}, 
which we may extend globally.  We split the appropriate material into two subsections -- subsections \ref{subsectconstant} 
and \ref{subsectcontinuous} will treat the constant and continuous cases for $A$ respectively.  

\subsection{$H^1$-norm stability estimate for the case of constant coefficient $A$} \label{subsectconstant}

Consider $A\equiv A_0$ on $\W$.  Then we may write
\begin{align*}
	\mA_0(w,v) &:=  (-A_0:D^2w,v) \\
	&= (-\div(A_0\grad w),v) \\
	&=  (A_0\grad w,\grad v)
\end{align*}
for any $w\in H^2(\W)\cap H_0^1(\W)$ and $v\in H_0^1(\W)$.  Clearly $\mA_0(\cdot,\cdot)$ is continuous and coercive on $H_0^1(\W)$ and $V_h$ with respect to the norm $\|\grad w\|_{L^2(\W)}$.  Define $\mLzh:V_h\to V_h$ by
\[
	(\mLzh w_h,v_h) = \mA_0(w_h,v_h) \quad\forall v_h\in V_h.
\]
Moreover, we can easily extend the domain of $\mLzh$ as a mapping from  $H_0^1(\W)$ to $H^{-1}(\W)$.  
By the finite element theory for elliptic problems, we obtain a discrete $H^1\to H^{-1}$ local stability estimate for $\mLzh$ shown in the following lemma.

\begin{lemma} \label{lemhm1constantlocal}
Let $x_0\in \W$, $R>0$ with $R>3h$.  Then for any $w_h\in V_h(B_R)$ we have
\begin{align}\label{eqnhm1constantlocal}
\|\grad w_h\|_{L^2(B_R)} \lss \|\mLzh w_h \|_{H_h^{-1}(B_{R})}.
\end{align}
where $B_R$ and is the ball centered at $x_0$ with radius $R$.  
\end{lemma}

\begin{proof}
Let $x_0\in\W$ and $w_h\in V_h(B_R)\setminus\{0\}$ which is nonempty since $R>3h$.  Note $w_h\in V_h$ and $w\equiv 0$ on $\W\setminus B_R$. We use the fact that $\mAze(\cdot,\cdot)$ is coercive to obtain
\begin{align}\label{eqn3.3}
\| \grad w_h \|_{L^2(B_R)}^2 &= \| \grad w_h \|_{L^2(\W)}^2 = (\grad w_h,\grad w_h) \\
&\lss \lambda(\grad w_h,\grad w_h)  \nonumber\\
&\lss (\mLzh w_h,w_h)  \nonumber\\
&\lss (\mLzh w_h,w_h)_{B_R} \nonumber \\
&\lss \|\mLzh w_h\|_{H_h^{-1}(B_R)}\|\grad w_h\|_{L^2(B_R)}. \nonumber
\end{align}
Dividing both sides by $\|\grad w_h\|_{L^2(B_R)}$ gives us \eqref{eqnhm1constantlocal}. The proof is complete.
\end{proof} 

\subsection{$H^1$-norm stability estimate for the case of continuous coefficient $A$} \label{subsectcontinuous}

Our goal for this subsection is to use Lemma \ref{lemhm1constantlocal} to show \eqref{eqnh1stabtemp}.  In order to achieve this, we first must take a new look at the operator $\mLh$.  Note we originally defined $\mLh$ as bounded linear operator from $H_h^2(\W)$ to $((V_h)^*,\|\cdot\|_{L_h^2})$.   Here the boundedness comes from Lemma 3.3 in \cite{AX:FN}, namely
\begin{align}
\|\mLh w\|_{L_h^2(D)} \lss \|w\|_{H^2_h(D)}.
\end{align} 
However, just like with $\mLzh$, we can also view $\mLh$ as an operator from $H_0^1(\W)$ to $H^{-1}(\W)$ through the following lemma:

\begin{lemma}\label{lemmahm1}
Let $D\subseteq \W$ be a subdomain.  Then for any $w\in H_h^2(\W)$, we have 
\begin{align} 
\|\mLh w\|_{H^{-1}(D)} &\lss \|\grad w\|_{L^2(\W)}, \label{eqnhm1-2}  \\
\|\mLh w\|_{H_h^{-1}(D)} &\lss \|\grad w\|_{L^2(\W)}. \label{eqnhm1-1}
\end{align}
\end{lemma}

\begin{proof}
We first consider the case where $w\in H^2(D)\cap H_0^1(D)$.  Note that since $w\in H^2(D)$, the term $\left< A[\grad w]\dv,v \right>_e$ vanishes for all $e\in\EI$.   Thus on $H^2(D)$ we have
\begin{align}
(\mLh w,v) = (-A:D^2 w,v) = (\mL w,v)
\end{align}
for any $v\in H_0^1(D)$. 

We wish to show
\begin{align}\label{eqnhm1}
(-A:D^2w,v)_D \lss \|A\|_{L^{\infty}(D)}\|\grad w\|_{L^2(D)}\|\grad v\|_{L^2(D)}.
\end{align}
for any $v\in H_0^1(D)$.  Since $C^{\infty}(D)\cap H_0^1(D)$ is dense in $H^2(D)\cap H_0^1(D)$, let $\phi_k\in C^{\infty}(D)\cap H_0^1(D)$ be a sequence such that $\phi^k\to w$ in $H^2(D)\cap H_0^1(D)$. H\"older's inequality gives us
\begin{align*}
(-a_{ij}\phi^k_{x_ix_j},v)_D \leq \|a_{ij}\|_{L^\infty(D)}\|\phi^k_{x_ix_j}v\|_{L^1(\W)} \leq \|a_{ij}\|_{L^\infty(D)}(|\phi^k_{x_ix_j}|,|v|)_D.
\end{align*}
We decompose $\phi^k_{x_ix_j} = \phi^{k,+}_{x_ix_j} - \phi^{k,-}_{x_ix_j}$ where $\phi^{k,+}_{x_ix_j} = \max\{0,\phi^k_{x_i,x_j}\}$ and $\phi^{k,-}_{x_ix_j} = \max\{0,-\phi_{x_i,x_j}\}$.  Thus $|\phi^k_{x_ix_j}| = \phi^{k,+}_{x_ix_j} + \phi^{k,-}_{x_ix_j}$.  Since $\phi^k_{x_ix_j}$ is continous, $\{\phi^k_{x_ix_j}>0\}\subseteq D$ is open. Hence we may integrate by parts to see
\begin{align*}
(\phi^{k,+}_{x_ix_j},|v|)_D &= \int_{\{\phi^k_{x_ix_j}>0\}} \phi^k_{x_ix_j}|v| \dx[x] = -\int_{\{\phi^k_{x_ix_j}>0\}}\phi^k_{x_i}|v|_{x_j} \dx[x] \\&\leq \|\phi^k_{x_i}\|_{L^2(D)}\|v_{x_j}\|_{L^2(D)}.
\end{align*}
The last inequality follows from $\big||v|_{x_j}\big| = |v_{x_j}|$.  We can show the same result for $\phi^{k,-}_{x_ix_j}$; therefore
\begin{align*}
(a_{ij}\phi^k_{x_ix_j},v)_D \leq 2 \|a_{ij}\|_{L^\infty(D)}\|\phi^k_{x_i}\|_{L^2(D)}\|v_{x_j}\|_{L^2(D)}
\end{align*}
Summing for $i,j=1,\ldots,d$, we have
\begin{align*}
(-A:D^2\phi^k,v)_D \leq 2d \|A\|_{L^\infty(D)}\|\grad \phi_k\|_{L^2(D)}\|\grad v\|_{L^2(D)}.
\end{align*}
Letting $k\to \infty$ we arrive at \eqref{eqnhm1}.

Hence we have shown the map $\mL:H^2(D)\cap H_0^1(D)\to H^{-1}(D)$ is a bounded map when $H^2(D)\cap H_0^1(D)$ is endowed with the strong $H^1$ topology.  Since $H^2(D)\cap H_0^1(D)$ is dense in $H_0^1(D)$ under this topology, we may continuously extend $\mL$ to a bounded map $\mL':H_0^1(D)\to H^{-1}(D)$ such that $\mL'\equiv \mL \equiv \mLh$ on $H^2(D)\cap H_0^1(D)$.  

We now wish to show $\mLh=\mL'$ on $H_h^2(D)$, that is
\begin{align}\label{eqn3.2.1.10}
(\mLh w,v) = (\mL'w,v) \quad\forall w\in H_h^2(D), v\in H_0^1(D).
\end{align}
To accomplish this, let $w\in H_h^2(D)$ and consider $w_\rho\in H^2(D)\cap H_0^1(D)$ such that $w_\rho\to w$ in $H_0^1(D)$.  Additionally consider $A\in C^1(\overline\W)$, then we have
\begin{align} \label{eqn3.2.1.11}
-A:D^2 w_\rho = -\div(A\grad w_\rho) + \div(A)\cdot\grad w_\rho.
\end{align}  
Let $v\in H_0^1(\W)$.  By \eqref{eqn3.2.1.11} and integration by parts we have
\begin{align}\label{eqn3.2.1.12}
\begin{split}
(\mL' w_\rho, v) = (\mL w_\rho,v) &= (-A:D^2 w_\rho,v) \\
&= (-\div(A\grad w_\rho),v) + (\div(A)\cdot\grad w_\rho,v) \\
&= (A\grad w_\rho,\grad v) + (\div(A)\cdot\grad w_\rho,v).
\end{split}
\end{align}
Since $w_\rho\to w$ in $H_0^1(D)$ we may pass the limit as $\rho\to 0$ in \eqref{eqn3.2.1.12} to obtain
\begin{align} \label{eqn3.2.1.13}
(\mL' w,v) = \lim_{\rho\to 0}(\mL' w_\rho,v) = (A\grad w,\grad v) + (\div(A)\cdot\grad w,v),
\end{align}
with estimate
\begin{align} \label{eqn3.2.1.13.1}
(\mL'w,v) \lss \|A\|_{L^{\infty}(\W)}\|\grad w\|_{L^2(D)}\|\grad v\|_{L^2(D)}.
\end{align}  
Since $w$ is $H^2$ on every $T\in\Th$, we may perform integration by parts on \eqref{eqn3.2.1.13} element-wise on every $T\in\Th$ and again apply \eqref{eqn3.2.1.11} to obtain
\begin{align}\label{eqn3.2.1.14}
\begin{split}
(\mL' w,v) &= (-A:D^2w,v) + \sum_{e\in\EI}\left<[A\grad w\dv],v \right> \\
&= (\mLh w,v)
\end{split}
\end{align} 
with estimate 
\begin{align} \label{eqn3.2.1.15}
 (\mLh w,v) \lss \|A\|_{L^{\infty}(\W)}\|\grad w\|_{L^2(D)}\|\grad v\|_{L^2(D)}
\end{align}
following from \eqref{eqn3.2.1.13.1}. 

We will now remove the differentiability condition on $A$.  
Since $C^1(\overline\W)$ is dense in $C^0(\overline\W)$ with the strong $C^0$-topology, \eqref{eqn3.2.1.15} implies (\ref{eqn3.2.1.14}-\ref{eqn3.2.1.15}) hold for $A\in C^0(\overline\W)$. Therefore, dividing \eqref{eqn3.2.1.15} by $\|\grad v\|_{L^2(\W)}$ and taking the supremem over all $v\in H_0^1(D)\setminus\{0\}$ yields \eqref{eqnhm1-2}. \eqref{eqnhm1-1} follows from setting $v=v_h\in V_h$ in \eqref{eqn3.2.1.15}.  The proof is complete.
\end{proof}

Next, we must show that locally $\mLh$ and $\mLzh$ are close in the discrete $H^{-1}$ norm.  This is shown in the following lemma.

\begin{lemma}\label{lemlocalclose}
For any $\delta>0$, there exists $R_\d>0$ and $h_\d>0$ such that for any $x_0\in\W$ with $A_0=A(x_0)$, for any $h\leq h_\d$ and $w\in H_h^2(\W)$ we have
\begin{align} \label{eqn3.2.1}
\| (\mLh-\mLzh)w \|_{H_h^{-1}(B_{R_\d})} \lss \d \|\grad w\|_{L^2(B_{R_\d})}.
\end{align}
where $B_{R_\d}:=B_{R_\d}(x_0)$.
\end{lemma}

\begin{proof}
Since $A$ is continuous on compact $\W$, it is uniformly continuous.  Thus for any $\d>0$ there is an $R_\d>0$ such that
\begin{align*}
\|A-A_0\|_{L^\infty(B_{R_\d})} \leq \d.
\end{align*}
Fix $w\in H_h^2(\W)$ and let $h_\d=\frac{1}{3}R_\d$ with $h\leq h_\d$ such that $V_h(B_{R_\d})$ is non-trivial. Note that the operator $\mLh-\mLzh$ has the same form as $\mLh$ but has $A-A_0$ instead of $A$.  Thus, we can repeat the proof of Lemma \ref{lemmahm1} with $A-A_0$ instead of $A$ and bounding this difference uniformly by $\d$ to obtain \eqref{eqn3.2.1}.  The proof is complete.
\end{proof}

Now we focus on a local $H^1\to H^{-1}$ stability estimate for $\mLh$.

\begin{lemma}\label{lemlocalcontinuous}
Let $x_0\in \W$.   There exists $R_1>0$ and $h_*>0$ such that for any $h<h_*$ we have 
\begin{align}\label{eqnlocalcontinuous}
\|\grad w_h\|_{L^2(B_1)} \lss \|\mLh w_h\|_{H_h^{-1}(B_1)}
\end{align}
for any $w_h\in V_h(B_1)$ where $B_1:=B_{R_1}(x_0)$.  
\end{lemma}

\begin{proof}
For any $\d>0$, let $R_1=R_{\d}$ and $h_* = \frac{1}{3}R_1$.  We apply Lemma \ref{lemlocalcontinuous} and Lemma \ref{lemhm1constantlocal} to $w_h\in V_h(B_1)$ for any $h<h_*$ to see
\begin{align} \label{eqn3.2.2}
\|\grad w_h\|_{L^2(B_1)} &\lss \|\mLzh w_h\|_{H_h^{-1}(B_1)} \\
&\leq \|(\mLh-\mLzh)w_h\|_{H_h^{-1}(B_1)} + \|\mLh w_h\|_{H_h^{-1}(B_1)} \nonumber \\
&\lss \delta \|\grad w_h\|_{L^2(B_1)} + \|\mLh w_h\|_{H_h^{-1}(B_1)}. \nonumber
\end{align}
Thus we choose $\d$, only dependent on $A$, sufficiently small such that we may move $\|\grad w_h\|_{L^2(B_1)}$ from the right side to the left side.  The proof is complete.
\end{proof}

 We now attempt to extend \eqref{eqnlocalcontinuous} to a global estimate using cutoff functions and a covering argument, but arrive at a G\aa rding-type estimate for now.

\begin{lemma}\label{lemglobalgarding} There is an $h_*>0$ such that for any $h<h_*$ and $w_h\in V_h$ we have
\begin{align}\label{eqnglobalgarding}
\|\grad w_h\|_{L^2(\W)} \lss \|\mLh w_h\|_{H_h^{-1}(\W)} + \|w_h\|_{L^2(\W)}.
\end{align}
\end{lemma}

\begin{proof} 
Let $x_0\in \W$ and let $h_*$, $R_1$, and $B_1$ be defined as in Lemma \ref{lemlocalcontinuous}. We first extend \eqref{eqnlocalcontinuous} to functions in $V_h$.  Let $w_h\in V_h$ and set $R_2=2R_1$ and $B_{R_2} := B_{R_2}(x_0)$. Let $\eta\in C^{\infty}(\W)$ be a cutoff function that satisfies
\begin{align} \label{eqnetaprop}
\quad\eta\big|_{B_1} = 1,\quad \eta\big|_{\W\setminus B_2} = 0,\quad \|\eta\|_{W^{m,\infty}} = \mathcal{O}(R_1^{-m})
\end{align}
for $m=0,1,2$.  
Note that $\eta w_h = w_h$ in $B_1$ and $I_h(\eta w_h)\in V_h(B_2)$.  By Lemmas 
 \ref{lemlocalcontinuous}, \ref{lemmahm1}, and \ref{lemsuperapprox} with $3h < R_1=d$, we obtain
\begin{align}\label{eqn3.2.4}
&\|\grad w_h \|_{L^2(B_1)} = \|\grad(\eta w_h)\|_{L^2(B_1)} \\
&\quad \leq \|\grad(\eta w_h-I_h(\eta w_h))\|_{L^2(B_1)} + \|\grad I_h(\eta w_h)\|_{L^2(B_1)}  \nonumber\\
&\quad \lss \frac{1}{R_1}\|w_h\|_{L^2(B_1)} + \|\grad I_h(\eta w_h)\|_{L^2(B_2)} \nonumber \\
&\quad \lss \frac{1}{R_1}\|w_h\|_{L^2(B_2)} + \|\mLh I_h(\eta w_h)\|_{H_h^{-1}(B_2)} \nonumber \\
&\quad \lss \frac{1}{R_1}\|w_h\|_{L^2(B_2)} + \|\mLh (\eta w_h)\|_{H_h^{-1}(B_2)} + \|\mLh(\eta w_h- I_h(\eta w_h))\|_{H_h^{-1}(B_2)} \nonumber \\
&\quad \lss \frac{1}{R_1}\|w_h\|_{L^2(B_2)} + \|\mLh (\eta w_h)\|_{H_h^{-1}(B_2)} + \|\grad(\eta w_h - I_h(\eta w_h))\|_{L^2(B_2)} 
\nonumber\\
&\quad \lss \frac{1}{R_1}\|w_h\|_{L^2(B_2)} + \|\mLh (\eta w_h)\|_{H_h^{-1}(B_2)} + \frac{1}{R_1}\|w_h\|_{L^2(B_2)} \nonumber\\
&\quad \lss \frac{1}{R_1}\|w_h\|_{L^2(B_2)} + \|\mLh (\eta w_h)\|_{H_h^{-1}(B_2)}. \nonumber
\end{align}
We now must remove $\eta$ from the $\mLh$ term.  To do this, we directly manipulate the weak form.  Let $v_h\in V_h(B_2)\setminus\{0\}$.  Since $\eta$ and $\grad\eta$ are continuous across any edge $e\in\EI$, we have
\begin{align}\label{eqn3.2.5}
(\mLh (\eta w_h),v_h) &=   (-A:D^2(\eta w_h),v_h) + \sum_{e\in\EI}\left<A[\grad(\eta w_h)\dv],v_h\right>_e \\
&= -(\eta A:D^2w_h+2A\grad \eta\cdot\grad w_h+w_hA:D^2\eta,v_h)  \nonumber\\
& \qquad+ \sum_{e\in\EI}\left<A[\grad w_h\dv],\eta v_h\right>_e \nonumber\\
&= (-A:D^2 w_h,\eta v_h)+ \sum_{e\in\EI}\left<A[\grad w_h\dv],\eta v_h\right>_e \nonumber \\
&\qquad + (-w_hA:D^2\eta,v_h) \nonumber \\
&\qquad + (-2A\grad\eta\cdot\grad w_h,v_h) \nonumber \\
&= (\mLh w_h,\eta v_h) + (-w_hA:D^2\eta,v_h) + (-2A\grad\eta\cdot\grad w_h,v_h) \nonumber \\
&:=I_1 + I_2 + I_3. \nonumber
\end{align}
We seek to bound $I_1$ and $I_2$ and $I_3$ independently.  We start with $I_1$.  Note $\mLh w_h\in V_h\subset H^{-1}(B_2)$.  Thus by \eqref{eqnhm1-2},  \eqref{normequiv}, and the Poincar\'e inequality we have 
\begin{align} \label{eqn3.2.6}
I_1 &= (\mLh w_h,\eta v_h) \leq \|\mLh w_h\|_{H^{-1}(B_2)}\|\grad(\eta v_h)\|_{L^2(B_2)} \\ 
&\lss \|\mLh w_h\|_{H_h^{-1}(B_2)}\left(\|\grad\eta\|_{L^\infty(B_2)}\|v_h\|_{L^2(B_2)}+ \|\eta\|_{L^\infty(B_2)}\|\grad v_h\|_{L^2(B_2)} \right) \nonumber \\
&\lss \frac{1}{R_1}\|\mLh w_h\|_{H_h^{-1}(B_2)}\|\grad v_h\|_{L^2(B_2)}. \nonumber
\end{align}
For $I_2$, we may apply the Poincar\'e inequality to obtain\begin{align}\label{eqn3.2.9}
I_2 \lss \frac{1}{R_1^2}\|w_h\|_{L^2(B_2)}\|\grad v_h\|_{L^2(B_2)}.
\end{align}
For $I_3$, using H\"older's inequality we have
\begin{align} \label{eqn3.2.10}
I_3 \lss \frac{1}{R_1}\|(\grad w_h)v_h\|_{L^1(B_2)}.
\end{align}
For $w\in H^1(B_2)$ and $v\in H_0^1(B_2)$, define $w^\pm$ by $w^+=\max\{0,w\}$ and $w^-=\max\{0,-w\}$ and $v^{\pm}$ similarly.  Since $w^\pm\in H^1(B_2)$ and $v^\pm\in H_0^1(B_2)$, we have
\begin{align}\label{eqn3.2.11}
(w_{x_i}^\pm,v^\pm)_{B_2} &= -(w^{\pm},v_{x_i}^\pm)_{B_2} \leq \|w^{\pm}\|_{L^2(B_2)}\|\grad v^\pm\|_{L^2(B_2)} \\
&\leq \|w\|_{L^2(B_2)}\|\grad v\|_{L^2(B_2)}. \nonumber
\end{align}
Since $|w|=w^+ + w^-$ and $|v|=v^+ + v^-$, \eqref{eqn3.2.11} implies 
\begin{align*}
\|(\grad w)v\|_{L^1(B_2)}\lss \|w\|_{L^2(B_2)}\|\grad v\|_{L^2(B_2)}.
\end{align*}
Thus
\begin{align} \label{eqn3.2.12}
I_3 \lss \frac{1}{R_1}\|w_h\|_{L^2(B_2)}\|\grad v_h\|_{L^2(B_2)}.
\end{align}
Hence, \eqref{eqn3.2.5}, \eqref{eqn3.2.6}, \eqref{eqn3.2.9}, and \eqref{eqn3.2.12} imply
\begin{align} \label{eqn3.2.13}
(\mLh (\eta w_h),v_h) \lss \left(\frac{1}{R_1}\|\mLh w_h\|_{H_h^{-1}(B_2)} + \frac{1}{R_1^2}\|w_h\|_{L^2(B_2)}\right)\|\grad v_h\|_{L^2(B_2)}
\end{align}
Dividing \eqref{eqn3.2.13} by $\|\grad v_h\|_{L^2(B_2)}$ and taking the supremum over all $v_h\in V_h(B_2)\setminus\{0\}$, we have 
\begin{align} \label{eqn3.2.14}
 \|\mLh (\eta w_h)\|_{H_h^{-1}(B_2)} \lss \frac{1}{R_1}\|\mLh w_h\|_{H_h^{-1}(B_2)} + \frac{1}{R_1^2}\|w_h\|_{L^2(B_2)}.
\end{align}
Therefore from \eqref{eqn3.2.4} and \eqref{eqn3.2.14} we obtain
\begin{align} \label{eqn3.2.15}
\|\grad w_h\|_{L^2(B_1)} \lss \frac{1}{R_1}\|\mLh w_h\|_{H_h^{-1}(B_2)} + \frac{1}{R_1^2}\|w_h\|_{L^2(B_2)}
\end{align}
for every $w_h\in V_h$.  We note that $R_1$ is not dependent on $h$, but the rather the continuity of $A$.  Thus we can cover $\overline{\W}$ with a finite number of balls and extend \eqref{eqn3.2.15} to a global estimate; namely,
\begin{align*}
\|\grad w_h\|_{L^2(\W)} \lss \|\mLh w_h\|_{H_h^{-1}(\W)} + \|w_h\|_{L^2(\W)}
\end{align*}
for all $w_h\in V_h$ and $h<h_*$ for some $h_*>0$ which is exactly \eqref{eqnglobalgarding}.  We point the reader to Lemma 3.4, Step 2 of \cite{AX:FN} for the details of the covering argument.  The proof is complete.
\end{proof}

We now wish to strip the $\|w_h\|_{L^2(\W)}$ term off of \eqref{eqnglobalgarding} to arrive at our $H^1\to H^{-1}$ 
stability result. We can easily do this for quadratic elements or greater since we know $\mLh$ is invertible for $r\geq 2$.  Here $r$ is the polynomial degree of $V_h^r=V_h$.  However we have not shown that $\mLh$ is invertible for linear elements.   

To continue we focus on the case $r=1$.   In this case $D_h^2w_h$ is identically zero, so we have 
\begin{align*}
(\mLh w_h,v_h) = \sum_{e\in\EI}\left< [A\grad w_h\dv],v_h \right>_e.
\end{align*}
To show $\mLh$ is invertible, we employ a nonstandard duality argument utilized in \cite{AX:FN,Feng2017SS,Feng2018SS}.  Define the discrete adjoint $\mLh^*:V_h\to V_h$ by 
\begin{align*}
(\mLh^*v_h,w_h) = (\mLh w_h,v_h) \quad\forall w_h,v_h\in V_h.   
\end{align*}
We note that since $V_h$ is finite dimensional, invertibility of $\mLh$ and $\mLh^*$ are equivalent.   In order to show $\mLh^*$ is invertible, we first show the following lemma.

\begin{lemma}\label{lemadjoint}
Let $r=1$.  There exists $h_{**}>0$ such that for any $h<h_{**}$ and $v_h\in V_h$ there holds  
\begin{align}\label{eqnadjstab}
\|\grad v_h\|_{L^2(\W)} \lss \|\mLh^*v_h\|_{H_h^{-1}(\W)}.
\end{align}
Moreover, both $\mLh^*$ and $\mLh$ are invertible on $V_h$.  
\end{lemma}

\begin{proof}
 We divide the proof into three steps. 

{\em Step 1: Local Estimates.} Let $x_0\in \W$ and $\d>0$.  Set $R_1$, $h_*$, and $B_1$ as in Lemma \ref{lemglobalgarding}.  Note $\mLzh^*=\mLzh$ since $\mLzh$ is self-adjoint. Let $v_h\in V_h(B_1)$.  By Lemma \ref{lemhm1constantlocal} and \eqref{eqn3.2.1.15} with coefficient matrix $A_0-A$ we have 
\begin{align}
\|\grad v_h\|_{L^2(B_1)} &\lss \|\mLzh^*v_h\|_{H_h^{-1}(B_1)} \\
&\leq \|(\mLzh^*-\mLh^*)v_h\|_{H_h^{-1}(B_1)} + \|\mLh^*v_h\|_{H_h^{-1}(B_1)} \nonumber \\
&\lss \sup_{w_h\in V_h\setminus\{0\}} \frac{((\mLzh^*-\mLh^*)v_h,w_h)}{\|\grad w_h\|_{L^2(B_1)}} + \|\mLh^*v_h\|_{H_h^{-1}(B_1)} \nonumber \\
&\lss \sup_{w_h\in V_h\setminus\{0\}} \frac{(\mLzh-\mLh) w_h,v_h)}{\|\grad w_h\|_{L^2(B_1)}} + \|\mLh^*v_h\|_{H_h^{-1}(B_1)} \nonumber \\
&\lss \sup_{w_h\in V_h\setminus\{0\}} \frac{\d \|\grad w_h\|_{L^2(B_1)}\|\grad v_h\|_{L^2(B_1)}}{\|\grad w_h\|_{L^2(B_1)}} + \|\mLh^*v_h\|_{H_h^{-1}(B_1)} \nonumber \\
&\lss \delta\|\grad v_h\|_{L^2(B_1)}+ \|\mLh^*v_h\|_{H_h^{-1}(B_1)}. \nonumber
\end{align}
Thus we can choose $\d$, independent of $h$, sufficiently small to achieve
\begin{align} \label{eqn3.2.20}
\|\grad v_h\|_{L^2(B_1)} \lss \|\mLh^*v_h\|_{H_h^{-1}(B_1)}
\end{align}

{\em Step 2: G\aa rding Inequality.} We now seek to replicate Lemma \ref{lemglobalgarding} for $\mLh^*$. Let $v_h\in V_h$.   Set $R_2$, $B_2$, and $\eta\in C^{\infty}$ as in Lemma \ref{lemglobalgarding}.  Then by Lemmata \ref{lemsuperapprox} and  \ref{lemmahm1} and \eqref{eqn3.2.20} we have
\begin{align} \label{eqn3.2.21}
&\|\grad v_h\|_{L^2(B_1)} = \|\grad (\eta v_h)\|_{L^2(B_1)} \\
&\quad \leq  \|\grad (\eta v_h)-\grad I_h(\eta v_h)\|_{L^2(B_1)} + \|\grad I_h(\eta v_h)\|_{L^2(B_2)}  \nonumber\\
&\quad\lss \frac{1}{R_1} \|v_h\|_{L^2(B_1)} +  \|\mLh^* I_h(\eta v_h)\|_{H_h^{-1}(B_2)} \nonumber \\
&\quad\lss \frac{1}{R_1} \|v_h\|_{L^2(B_2)} +  \|\mLh^*( I_h(\eta v_h)-\eta v_h)\|_{H_h^{-1}(B_2)} +  \|\mLh^* (\eta v_h)\|_{H_h^{-1}(B_2)} \nonumber \\
&\quad\lss \frac{1}{R_1} \|v_h\|_{L^2(B_2)} +  \|\grad (\eta v_h)-\grad I_h(\eta v_h)\|_{L^2(B_2)} +  \|\mLh^* (\eta v_h)\|_{H_h^{-1}(B_2)} \nonumber \\
&\quad\lss \frac{1}{R_1} \|v_h\|_{L^2(B_2)} + \|\mLh^* (\eta v_h)\|_{H_h^{-1}(B_2)}. \nonumber
\end{align}
Let $w_h\in V_h(B_2)\setminus\{0\}$.  Then
\begin{align} \label{eqn3.2.22}
(\mLh^*(\eta v_h),w_h) = (\mLh w_h,\eta v_h) = \sum_{e\in\EI}\left<[A\grad w_h\dv],\eta v_h\right>.
\end{align}  
Since $\grad\eta, \eta,$ and $w_h$ are continuous across any edge $e\in\EI$, then $[\grad w_h]\eta = [\grad (\eta w_h)]$.  Thus \eqref{eqn3.2.22} implies
\begin{align}\label{eqn3.2.23}
(\mLh^*(\eta v_h),w_h) &= \sum_{e\in\EI}\left<[A\grad (\eta w_h)\dv], v_h\right> = (\mLh (\eta w_h),v_h) = (\mLh^* v_h,\eta w_h) \\
&\lss \|\mLh^* v_h\|_{H^{-1}(B_2)}\|\grad(\eta w_h)\|_{L^2(B_2)}  \nonumber\\
&\lss \frac{1}{R_1}\|\mLh^* v_h\|_{H_h^{-1}(B_2)}\|\grad w_h\|_{L^2(B_2)}.\nonumber
\end{align}
Hence using \eqref{eqn3.2.21} and \eqref{eqn3.2.23} we obtain
\begin{align}\label{eqn3.2.24}
\|\grad v_h\|_{L^2(B_1)} \lss \frac{1}{R_1}\|\mLh^* v_h\|_{H_h^{-1}(B_2)} + \frac{1}{R_1} \|v_h\|_{L^2(B_2)}.
\end{align}
We note that $R_1$ depends on the continuity of $A$ and not $h$.  Thus using a covering argument we can extend \eqref{eqn3.2.24} to a G\aa rding-type inequality on $\W$; namely, 
\begin{align}\label{eqn3.2.25}
\|\grad v_h\|_{L^2(\W)} \lss \|\mLh^* v_h\|_{H_h^{-1}(\W)} + \|v_h\|_{L^2(\W)}.
\end{align}

{\em Step 2: Non-standard Duality Argument.}  We now perform a duality argument on $\mLh^*$ using $\mL$. Let 
\begin{align*}
X = \{g\in L^2(\W): \|g\|_{L^2(\W)} = 1\}.
\end{align*}
We note that $X$ is precompact in $H^{-1}(\W)$.  Define $W\subset H^2(\W)\cap H_0^1(\W)$ by 
\begin{align*}
W = \{\mL^{-1}g:g\in X\}.
\end{align*}
By \eqref{eqn1.3}, $\mL^{-1}$ is a continuous map from $H^{-1}(\W)$ to $H_0^1(\W)$, thus $W$ is precompact in $H_0^1(\W)$.  Let $\e>0$.  Then by Lemma 2 of \cite{MC:SW}, there exists $h_2=h_2(\e,W)>0$ such that for any $w\in W$ and $h\leq h_2$ there is a $w_h\in V_h$ such that 
\begin{align} \label{eqn3.2.26}
\|w-w_h\|_{H^1(\W)} \leq \e.  
\end{align}
By the reverse triangle equality and \eqref{eqn1.3}, we have 
\begin{align*}
\|\grad w_h\|_{L^2(\W)} \leq \|\grad w\|_{L^2(B_1)} + \e \lss \|g\|_{H^{-1}(\W)} + \e \leq \|g\|_{L^2(\W)} + \e \lss 1
\end{align*}

For $g\in X$, set $w_g = \mL^{-1}g\in W$.  Therefore by Lemma \ref{lemmahm1}, for any $w_h\in V_h$ satisfying \eqref{eqn3.2.26} we have
\begin{align} \label{eqn3.2.27}
(v_h,g) &= (\mLh w_g,v_h) = (\mLh^*v_h,w_g) = (\mLh^*v_h,w_h) + (\mLh^*v_h,w_g-w_h)  \\
&= (\mLh^*v_h,w_h) + (\mLh(w_g-w_h),v_h)  \nonumber\\
&\lss \|\mLh^*v_h\|_{H_h^{-1}(\W)} \|\grad w_h\|_{L^2(\W)} + \|w_g-w_h\|_{H^1(\W)}\|\grad v_h\|_{L^2(\W)} \nonumber \\
&\lss \|\mLh^*v_h\|_{H_h^{-1}(\W)} +\e\|\grad v_h\|_{L^2(\W)}. \nonumber
\end{align}
Taking the supremum of \eqref{eqn3.2.27} over all $g\in X$ yields
\begin{align} \label{eqn3.2.28}
\|v_h\|_{L^2(\W)} \lss  \|\mLh^*v_h\|_{H_h^{-1}(\W)} +\e\|\grad v_h\|_{L^2(\W)}.
\end{align}
Thus by taking $\e$, independent of $h$, sufficiently small and setting $h_{**} = \min\{h_*,h_2\}$ we combine \eqref{eqn3.2.25} and \eqref{eqn3.2.28} to obtain 
\begin{align*}
\|\grad v_h\|_{L^2(\W)} \lss \|\mLh^*v_h\|_{H_h^{-1}(\W)}.
\end{align*}
which is \eqref{eqnadjstab}.  To show that $\mLh^*$ is invertible, we see if $\mLh^*v_h=0$, then \eqref{eqnadjstab} immediately implies $\|\grad v_h\|_{L^2(\W)} = 0$ which can only happen if $v_h = 0$.  Therefore  $\mLh^*$ is injective and thus invertible since $V_h$ is finite dimensional.  Moreover $\mLh$ is also invertible.  
The proof is complete.
\end{proof}

We can now strip the $\|w_h\|_{L^2(\W)}$ term off of \eqref{eqnglobalgarding} which will yeild our $H^1$ stability result.  To do so, we apply a proof by contradiction technique found in \cite[Lemma 9.17]{Sp:DT}.  
\begin{theorem}\label{thmmainstab}
There exists $h_{**}>0$ such that 
\begin{align} \label{eqnh1stab}
\|\grad w_h\|_{L^2(\W)} \lss \|\mLh w_h\|_{H_h^{-1}(\W)}
\end{align}
for all $h<h_{**}$ and $w_h\in V_h$. 
\end{theorem}

\begin{proof}
Let $r \geq 1$ and choose $h_{**}=h_*>0$ as in Lemma \ref{lemglobalgarding} for $r\geq 2$ or $h_{**}$ as in Lemma  \ref{lemadjoint} for $r=1$.  Suppose for the sake of contradiction there is a sequence of $w_h^k\in V_h$ such that $\|w_h^k\|_{L^2(\W)} = 1$ and $\|\mLh w_h^k\|_{H_h^{-1}(\W)}\to 0$ as $k\to \infty$.  By Lemma \ref{lemglobalgarding}, we have $\|w_h^k\|_{H^1(\W)}$ uniformly bounded in $k$.  Since $V_h$ is finite dimensional, there exists $w_h^*\in V_h$ such that $w_h^k\wto w_h^*$ weakly in $H^1(\W)$.  Thus $w_h^k\to w_h^*$ strongly in $L^2(\W)$ and $\|w_h^*\|_{L^2(\W)}=1$.  Since $\mLh$ is linear, we also have 
\begin{align*}
0\leq \|\mLh w_h^*\|_{H_h^{-1}(\W)}\leq \liminf_{k\to\infty}\|\mLh w_h^k\|_{H_h^{-1}(\W)}  = 0.
\end{align*}
Thus $\|\mLh w_h^*\|_{H_h^{-1}(\W)} = 0$ and from that we know $\mLh w_h^*= 0$.  Since $\mLh$ is invertible on $V_h$ for $r\geq 2$ by \cite{AX:FN} and for $r=1$ by Lemma \ref{lemadjoint}, $w_h^*=0$ which contradicts $\|w_h^*\|_{L^2(\W)}=1$.  Thus \eqref{eqnh1stab} holds for $r\geq 1$.  
The proof is complete.  
\end{proof}

Using \eqref{eqnh1stab}, we can build a Ce\'a-type lemma and thus an optimal error estimate 
for $\|u-u_h\|_{H^1(\W)}$.

\begin{theorem}
Let $u_h \in V_h$ and  $u\in H^2(\W)\cap H_0^1(\W)$  be the solutions to \eqref{eqnc1fe} and \eqref{eqn1.1} respectively.  Then there holds 
\begin{align}\label{eqnceah1}
\|u-u_h\|_{H^1(\W)} \lss \inf_{w_h\in V_h}\|u-w_h\|_{H^1(\W)}.
\end{align}
Moreover, if $u\in H^s(\W)$ for $s\geq 2$ we have
\begin{align}
\|u-u_h\|_{H^1(\W)} &\lss h^l\|u\|_{H^s(\W)}  \label{eqnerror1}
\end{align}
for $l=\min\{r+1,s\}$.
\end{theorem}

\begin{proof}
Let $w_h\in V_h$.  Since $\mLh$ is consistent, we have the usual Galerkin orthogonality; namely,
\begin{align} \label{eqn3.2.17}
(\mLh (u-u_h),v_h) = (\mLh u,v_h) - (\mLh u_h,v_h) = (f,v_h)-(f,v_h) = 0
\end{align}  
for any $v_h\in V_h$.  
By \eqref{eqnh1stab}, \eqref{eqnhm1-2}, and \eqref{eqn3.2.17} we have
\begin{align} \label{eqn3.2.18}
\|\grad(u_h-w_h)\|_{L^2(\W)} &\lss \|\mLh(u_h-w_h)\|_{H_h^{-1}(\W)}= \sup_{v_h\in V_h}\frac{(\mLh (u_h-w_h),v_h)}{\|\grad v_h\|_{L^2(\W)}} \\
&\lss \sup_{v_h\in V_h}\frac{(\mLh (u-w_h),v_h)}{\|\grad v_h\|_{L^2(\W)}} 
\lss \|\grad(u-w_h)\|_{L^2(\W)}. \nonumber
\end{align} 
\eqref{eqnceah1} then follows from an application of the triangle inequality and using \eqref{eqn3.2.18}. 
 Choosing $w_h=I_h u$ and using the standard interpolation estimates we obtain \eqref{eqnerror1}. The proof is complete.  
\end{proof}


\pagestyle{myheadings}
\thispagestyle{plain}

\bibliographystyle{abbrv}

\end{document}